\documentclass[11pt]{amsart}
\usepackage{verbatim, graphicx, rotating}
\usepackage{mdwlist}
\usepackage{enumerate,mdwlist}
\usepackage{url}
\usepackage{amsthm}
\usepackage{hyperref}

\usepackage {amssymb}

\input xy
\xyoption{all}
\input epsf

\newlength{\tabwidth}
\newlength{\tabheight}
\newlength{\tabrule}
\newlength{\tabwidthx}
\newlength{\tabheightx}

\def\gentabbox#1#2#3#4{\vbox to \tabheight{\setlength{\tabrule}{#3}%
  \setlength{\tabwidthx}{#1\tabwidth}\addtolength{\tabwidthx}{\tabrule}%

\setlength{\tabheightx}{#2\tabheight}\addtolength{\tabheightx}{-\tabheight}%
  \hbox to #1\tabwidth{%
    \hspace{-0.5\tabrule}\rule{\tabrule}{#2\tabheight}\hspace{-\tabrule}%
    \vbox to #2\tabheight{\hsize=\tabwidthx%
      \vspace{-0.5\tabrule}\hrule width\tabwidthx height\tabrule%
      \vspace{-0.5\tabrule}\vfil%
      \hbox to \tabwidthx{\hss#4\hss}%
        \vfil\vspace{-0.5\tabrule}%
      \hrule width\tabwidthx height\tabrule\vspace{-0.5\tabrule}}%
    \hspace{-\tabrule}\rule{\tabrule}{#2\tabheight}\hspace{-0.5\tabrule}}%
  \vspace{-\tabheightx}}}
\def\genblankbox#1#2{\vbox to \tabheight{\vfil\hbox to
#1\tabwidth{\hfil}}}
\def\tabbox#1#2#3{\gentabbox{#1}{#2}{0.4pt}{\strut #3}}

\catcode`\:=13 \catcode`\.=13 \catcode`\;=13 
\catcode`\>=13 \catcode`\^=13
\def:#1\\{\hbox{$#1$}}
\def.#1{\tabbox{1}{1}{$#1$}}
\def>#1{\tabbox{2}{1}{$#1$}}
\def^#1{\tabbox{1}{2}{$#1$}}
\def;{\genblankbox{1}{1}\relax}
\catcode`\:=12 \catcode`\.=12 \catcode`\;=12 
\catcode`\>=12 \catcode`\^=7

\newcommand{\field}{\mathbb}
\newcommand{\liealgebra}{\mathfrak}
\newcommand{\la}{\liealgebra}

\newcommand{\C}{{\field C}}

\newcommand{\Z}{{\field Z}}
\newcommand{\N}{{\field N}}

\renewcommand{\b}{\liealgebra b}

\newcommand{\n}{{\la n}}

\newcommand{\ga}{\alpha}

\newtheorem{prop}{Proposition}[section]

\newtheorem{lemma}[prop]{Lemma}
\newtheorem{theorem}[prop]{Theorem}

\newtheorem{proposition}[prop]{Proposition}

\theoremstyle{definition}

\newtheorem{remark}[prop]{Remark}
\newtheorem{example}[prop]{Example}

\newtheorem{definition}[prop]{Definition}

\newcommand{\frt}{\mathfrak{t}}

\newcommand{\bbN}{\mathbb{N}}

\newcommand{\bbP}{\mathbb{P}}

\begin{document}
\title[Richardson varieties stable under spherical Levi subgroups]
{Schubert calculus of Richardson varieties stable under spherical Levi subgroups}

\author{Benjamin J. Wyser}
\date{\today}


\begin{abstract}
We observe that the expansion in the basis of Schubert cycles for $H^*(G/B)$ of the class of a Richardson variety stable under a spherical Levi subgroup is described by a theorem of Brion.  Using this observation, along with a combinatorial model of the poset of certain symmetric subgroup orbit closures, we give positive combinatorial descriptions of certain Schubert structure constants on the full flag variety in type $A$.  Namely, we describe $c_{u,v}^w$ when $u$ and $v$ are inverse to Grassmannian permutations with unique descents at $p$ and $q$, respectively.  We offer some roughly stated conjectures for similar rules in types $B$ and $D$, associated to Richardson varieties stable under spherical Levi subgroups of $SO(2n+1,\C)$ and $SO(2n,\C)$, respectively.
\end{abstract}

\maketitle

\section{Introduction}
Suppose that $G$ is a complex reductive algebraic group, with $B,B^- \subseteq G$ opposite Borel subgroups intersecting in the maximal torus $T$.  Let $W = N_G(T)/T$ be the Weyl group.  For each $w \in W$, there exists a \textit{Schubert class} $S_w = [\overline{B^-wB/B}] \in H^*(G/B)$.  It is well-known that the classes $\{S_w\}_{w \in W}$ form a $\Z$-basis for $H^*(G/B)$.  As such, for any $u,v \in W$, we have
\[ S_u \cdot S_v = \displaystyle\sum_{w \in W} c_{u,v}^w S_w \]
in $H^*(G/B)$, for uniquely determined integers $c_{u,v}^w$.  These integers are the \textit{Schubert structure constants}.

The structure constants are known to be non-negative for geometric reasons, and are readily computable.  However, it is a long-standing open problem, even in type $A$, to give a \textit{positive} (i.e. subtraction-free) formula for $c_{u,v}^w$ in terms of the elements $u$, $v$, and $w$.

In type $A$, where $W=S_n$, there are numerous partial results which give positive formulas for structure constants $c_{u,v}^w$ in special cases.  Perhaps most notably, when $u,v$ are ``Grassmannian" permutations (each having a unique descent in the same place), the Schubert classes $S_u,S_v \in H^*(G/B)$ are pulled back from Schubert classes in the cohomology of a Grassmannian, and their products are determined by the classical Littlewood-Richardson rule, or by the equivalent ``puzzle rule" of \cite{Knutson-Tao-Woodward-04}.  Other examples include
\begin{itemize}
	\item Monk's rule (\cite{Monk-59}), which describes structure constants $c_{u,s_i}^w$ with $u \in W$ any permutation, and $s_i = (i,i+1)$ a simple transposition;
	\item An analogue of Pieri's rule for Grassmannians, which generalizes Monk's rule.  The formula determines $c_{u,v}^w$ when $u \in W$ is any permutation, and $v$ is a Grassmannian permutation of a certain ``shape" (\cite{Sottile-96});
	\item A rule due to M. Kogan (\cite{Kogan-01}) which describes $c_{u,v}^w$ when $u$ is a Grassmannian permutation with unique descent at $k$, and $v$ is a permutation all of whose descents are in positions at most $k$.  (Note that this generalizes the Littlewood-Richardson rule mentioned above.)
	\item A rule due to I. Coskun (\cite{Coskun-09}), which gives a positive description of structure constants in the cohomology ring of a two-step flag variety in terms of ``Mondrian tableaux".  (A manuscript on an extension of this rule to arbitrary partial flag varieties, currently available on I. Coskun's webpage, is described there as ``under revision"\footnote{Per \url{http://www.math.uic.edu/~coskun/}, as of September 4, 2012, the paper is described as follows:  ``Currently under revision. This is a preliminary version of a Littlewood-Richardson rule for arbitrary partial flag varieties. Any comments, corrections and suggestions are welcome."}.)
\end{itemize}

Although computing an arbitrary Schubert constant positively remains a difficult problem, it can be easier to compute certain Schubert constants which are somehow special with respect to a Levi subgroup of $G$.  For example, in \cite{Richmond-12} the problem of computing structure constants associated to ``Levi-movable" intersections in an arbitrary flag manifold is reduced to (easier) computations in the cohomology of various $G/P$ with $P$ maximal parabolic.  The main result of this paper (Theorem \ref{thm:structure_constants}) is another special case rule in type $A$, along the lines of those mentioned above, which gives a positive description of Schubert constants $c_{u,v}^w$ associated to Richardson varieties stable under the action of a special class of Levi subgroup.  Applying the same general observation which allows us to deduce this rule to other classical groups, we also conjecture similar special-case formulas for Schubert constants in types $B$ and $D$.

The aforementioned general observation is as follows:  Suppose that $P$ is a standard parabolic subgroup containing $B$, with opposite parabolic $P^-$.  The intersection $P \cap P^-$ is a common Levi factor $L$ of both parabolics.  Suppose further that $L$ is \textit{spherical}, i.e. that it acts with finitely many orbits on $G/B$.  Then if $X_u = \overline{BuB/B}$ is a Schubert variety stable under $P$, and $X^v = \overline{B^-vB/B}$ is an opposite Schubert variety stable under $P^-$, then the \textit{Richardson variety} $X_u^v := X_u \cap X^v$ is stable under $L$.  Since $L$ has finitely many orbits on $G/B$, it of course has finitely many orbits on $X_u^v$, and so there is a dense $L$-orbit on $X_u^v$.  Thus the Richardson variety $X_u^v$ is the closure of an $L$-orbit.

Since the intersection $X_u \cap X^v$ is reduced and proper, the class $[X_u^v] \in H^*(G/B)$ is precisely the product
\begin{equation}\label{eqn:classes_of_richardson_varieties}
	[X_u] \cdot [X^v] = [X^{w_0u}] \cdot [X^v] = S_{w_0u} \cdot S_v,
\end{equation}
so knowing the class of a Richardson variety in the Schubert basis amounts to knowing some Schubert constants.  As it turns out, there is a theorem (Theorem \ref{thm:brion}) due to M. Brion on expressing the class of any spherical subgroup orbit closure as a sum of Schubert cycles.  The rule is in terms of weighted paths in the ``weak order graph" for the set $L \backslash G/B$ of $L$-orbits on $G/B$.  By what we have observed thus far, this theorem applies directly to Richardson varieties stable under spherical Levi subgroups.

To turn these simple observations into a useful combinatorial rule, however, there is a bit of work to be done.  First, one must determine exactly which Richardson varieties are stable under $L$.  This turns out to be an easy task --- they are of the form $X_u^v$ where $u$ (respectively, $v$) is a maximal (respectively, minimal) length coset representative of $W_P \backslash W$, with $W_P$ the parabolic subgroup of $W$ associated to $P$.  Next, one must have a concrete combinatorial model of the set $L \backslash G/B$, as well as an understanding of its weak order in terms of that model.  Finally, one must be able to correctly match up Richardson varieties stable under $L$ with the appropriate $L$-orbit closures.  That is, given such a Richardson variety, one must be able to explicitly compute the combinatorial invariant for the $L$-orbit closure with which the Richardson variety coincides.

The pairs $(G,L)$ with $G$ a complex simple group and $L$ a spherical Levi subgroup of $G$ are classified, see \cite[Theorem 4.1]{Brundan-98}.  Those pairs for which $G$ is a \textit{classical} group are
\begin{enumerate}
	\item $(SL(p+q,\C), S(GL(p,\C) \times GL(q,\C)))$
	\item $(SO(2n+1,\C), \C^* \times SO(2n-1,\C))$
	\item $(SO(2n+1,\C), GL(n,\C))$
	\item $(Sp(2n,\C), \C^* \times Sp(2n-2,\C))$
	\item $(Sp(2n,\C), GL(n,\C))$
	\item $(SO(2n,\C), \C^* \times SO(2n-2,\C))$
	\item $(SO(2n,\C), GL(n,\C))$
\end{enumerate}
Cases (5) and (7) are treated in the paper \cite{Wyser-11b}.  In hindsight, the results of that paper are now seen to be specific examples of the more general observations made to this point.

In this paper, we carry out the steps described above for case (1), the pair $(SL(p+q,\C),S(GL(p,\C) \times GL(q,\C)))$.  Noting that each element of $GL(p,\C) \times GL(q,\C)$ can be written as an element of $S(GL(p,\C) \times GL(q,\C))$ times a scalar, and that this scalar acts trivially on the type $A$ flag variety, we see that the orbits of $S(GL(p,\C) \times GL(q,\C))$ on $SL(p+q,\C)/B$ coincide with the orbits of $GL(p,\C) \times GL(q,\C)$ on $GL(p+q,\C)/B$.  Thus we simplify notation a bit by considering instead the pair $(GL(p+q,\C),GL(p,\C) \times GL(q,\C))$.

After handling the type $A$ case, we conclude by offering some roughly stated conjectures related to cases (2) and (6), and some passing thoughts on cases (3) and (4).

\subsection*{Acknowledgements}
The author thanks William A. Graham, his research advisor at the University of Georgia, for his guidance throughout the author's doctoral thesis project, which led to the discovery of the results presented here.  The author further thanks Professor Graham for his assistance and advice in reading and editing early versions of this manuscript.

The author also thanks Laurent Manivel, whose helpful comments allowed for simplification of the proofs given in an earlier version of this paper, and inspired the conjectures and comments in Section 4.

Finally, the author thanks Allen Knutson and Alexander Yong for helpful email exchanges, and two anonymous referees for their careful readings and helpful suggestions.

\section{Preliminaries}\label{sect:prelims}
\subsection{Notation}
We establish some standard notation:

$[n]$ shall denote the set $\{1,\hdots,n\}$.

The long element of a Weyl group $W$ will always be denoted $w_0$.  When $W=S_n$, $w_0$ is the permutation given by $i \leftrightarrow n+1-i$ for all $i$.

When discussing Weyl groups outside of type $A$, we will be using signed permutations, i.e. bijections $\sigma$ on the set $\{\pm 1,\hdots,\pm n\}$ having the property that $\sigma(-i) = -\sigma(i)$ for each $i$.  Such a signed permutation will be denoted in one-line notation with bars over some of the numbers to indicate negative values.  For example, $1 \overline{3} 4 \overline{2}$ denotes the signed permutation sending $1$ to $1$, $2$ to $-3$, $3$ to $4$, and $4$ to $-2$.

The closed points of the varieties $G/B$ correspond to flags, which we denote $F_{\bullet}$, shorthand for the flag
\[ \{0\} = F_0 \subset F_1 \subset \hdots \subset F_{n-1} \subset F_n = \C^n, \]
with $\dim(F_i) = i$.  In type $A$, the flag variety $GL(n,\C)/B$ parametrizes all complete flags on $\C^n$.
 
\subsection{Permutations and the Bruhat order}
For any permutation $w \in S_n$, and for each $(i,j) \in [n] \times [n]$, define
\[ r_w(i,j) = \#\{k \leq i \ \vert \ w(k) \leq j\}. \]

\begin{definition}
We refer to the matrix $(e_{i,j}) = (r_w(i,j))$ as the \textbf{rank matrix} for the permutation $w$.
\end{definition}

We give two equivalent definitions of the Bruhat order on $S_n$.  We will make use of both definitions later.  That these two definitions are equivalent to each other (and to the ``usual" definition of the Bruhat order) is standard --- see \cite{Deodhar-77} or \cite[\S 10.5]{Fulton-YoungTableaux}.
\begin{definition}\label{def:bruhat-1}
The \textbf{Bruhat order} on $S_n$ is the partial order defined as follows:  $u \leq v$ if and only if
\[ r_u(i,j) \geq r_v(i,j) \text{ for all } i,j. \]
\end{definition}

\begin{definition}\label{def:bruhat-2}
Here is an alternative definition of the Bruhat order on $S_n$:  $u \leq v$ if and only if for any $i \in [n]$, when $\{u(1),\hdots,u(i)\}$ and $\{v(1),\hdots,v(i)\}$ are each arranged in ascending order, each element of the first set is less than or equal to the corresponding element of the second set.
\end{definition}

\subsection{Schubert varieties and Richardson varieties}\label{sect:schubert_defs}
The facts of this section are all standard, and can be found in, e.g., \cite{Fulton-YoungTableaux} and/or \cite{Brion-05}.

In type $A$, where $W = S_n$ and $G/B$ is the variety of complete flags on $\C^n$, given a permutation $u \in W$, the Schubert cell $X_u^0 = BuB/B$ consists of the following points:
\[ X_u^0 = \{F_{\bullet} \mid \dim(F_i \cap E_j) = r_u(i,j) \ \forall i,j\}, \]
where for each $j$, $E_j$ denotes the span of the first $j$ standard basis vectors, $\C \langle e_1,\hdots,e_j \rangle$.

The Schubert variety $X_u = \overline{BuB/B}$ is then 
\[ X_u = \{F_{\bullet} \mid \dim(F_i \cap E_j) \geq r_u(i,j) \ \forall i,j\}. \]

Similarly, for the permutation $v \in S_n$, the opposite Schubert cell $X^v_0 = B^-vB/B$ consists of the following points:
\[ X^v_0 = \{F_{\bullet} \mid \dim(F_i \cap \widetilde{E_j}) = r_{w_0v}(i,j) \ \forall i,j\}, \]
where for each $j$, $\widetilde{E_j}$ denotes the span of the \textit{last} $j$ standard basis vectors, $\C \langle e_n,\hdots,e_{n-j+1} \rangle$.

The opposite Schubert variety $X^v = \overline{B^-vB/B}$ is then
\[ X^v = \{F_{\bullet} \mid \dim(F_i \cap \widetilde{E_j}) \geq r_{w_0v}(i,j) \ \forall i,j\}. \]

We have the following standard facts about Richardson varieties:
\begin{prop}\label{prop:basic_richardson_varieties}
For $u,v \in W$, $X_u^v := X_u \cap X^v$ is non-empty if and only if $u \geq v$.  In this event, the intersection $X_u \cap X^v$ is proper and reduced, and has dimension $l(u) - l(v)$.
\end{prop}

We now note which Schubert varieties are stable under $P$, and which opposite Schubert varieties are stable under $P^-$.  Via the obvious bijection
\[ B \backslash G/P \leftrightarrow P \backslash G/B \]
given by $BgP \leftrightarrow Pg^{-1}B$, and using the standard fact that the $B$-orbits on $G/P$ are of the form 
\[ \{BwP/P\}_{wW_P \in W/W_P}, \] 
we see that the $P$-orbits on $G/B$ are of the form 
\[ \{PwB/B\}_{W_Pw \in W_P \backslash W}. \]

These $P$-orbits are of course $B$-stable, and each is a union of Schubert cells $X_w^0$ with $w$ running over all representatives of a single coset of $W_P \backslash W$.  Thus $P$-orbit closures on $G/B$ are Schubert varieties of the form $X_u$ where $u$ is a maximal length coset representative of $W_P \backslash W$.  These are precisely the Schubert varieties which are stable under $P$.

Arguing identically, one sees that the $P^-$-orbit closures on $G/B$ are translated Schubert varieties $w_0X_{w_0v}$ with $w_0v$ a maximal length element of $W_P \backslash W$.  These are precisely opposite Schubert varieties $X^v$ with $v$ a \textit{minimal} length coset representative of $W_P \backslash W$.  These are the opposite Schubert varieties stable under $P^-$.

This adds up to the following easy
\begin{proposition}\label{prop:richardsons-stable-under-l}
	Suppose that $P,P^-$ are opposite parabolics, intersecting in their common Levi factor $L$.  If $u$ is a maximal length coset representative of $W_P \backslash W$, $v$ is a minimal length coset representative of $W_P \backslash W$, and $u \geq v$, then the Richardson variety $X_u^v$ is stable under $P \cap P^- = L$.
\end{proposition}

\subsection{The weak order on spherical subgroup orbits}
Let $G$ be any complex, reductive algebraic group, with $H$ a spherical subgroup of $G$.  (Recall that this means that $H$ has finitely many orbits on $G/B$.)

The orbits of $H$ on $G/B$ are partially ordered by closure containment:  $Q_1 \leq Q_2 \Leftrightarrow \overline{Q_1} \subseteq \overline{Q_2}$.  We shall refer to this order as the ``closure order" or the ``full closure order".  A weaker order, which we call the ``weak order" or ``weak closure order", can be defined as follows:  For any simple root $\ga \in \Delta(G,T)$, let $P_{\alpha}$ denote the standard minimal parabolic subgroup of type $\ga$, and let
\[ \pi_{\ga}: G/B \rightarrow G/P_{\ga} \]
be the natural projection.  This is a locally trivial fiber bundle with fiber $P_{\ga}/B \cong \bbP^1$.  Given any $H$-orbit $Q$, one may consider the set $Z_{\ga}(Q):=\pi_{\ga}^{-1}(\pi_{\ga}(Q))$.  Because the map $\pi_{\ga}$ is $H$-equivariant, $Z_{\ga}(Q)$ is stable under $H$.  Assuming $H$ is connected, $Z_{\ga}(Q)$ is also irreducible, and so it contains a dense $H$-orbit.  (In the event that $H$ is disconnected, one notes that the component group of $H$ acts transitively on the irreducible components of $Z_{\ga}(Q)$, and from this it again follows that $Z_{\ga}(Q)$ has a dense $H$-orbit.)  We denote this dense orbit by $s_{\ga} \cdot Q$.

If $\dim(\pi_{\ga}(Q)) < \dim(Q)$, then $s_{\ga} \cdot Q = Q$.  However, if $\dim(\pi_{\ga}(Q)) = \dim(Q)$, then $s_{\ga} \cdot Q = Q'$ for some $Q' \neq Q$ with $\dim(Q') = \dim(Q) + 1$.

\begin{definition}
The \textbf{weak closure order} (or simply the \textbf{weak order}) is the partial order on $H$-orbits generated by relations of the form $Q \prec Q' \Leftrightarrow Q' = s_{\ga} \cdot Q \neq Q$ for some $\ga \in \Delta(G,T)$.
\end{definition}

Note that we may just as well speak of the weak order on orbit \textit{closures}.  Supposing that $Y,Y'$ are the closures of orbits $Q,Q'$, respectively, we say that $Y' = s_{\ga} \cdot Y$ if and only if $Q' = s_{\ga} \cdot Q$, if and only if $Y' = \pi_{\ga}^{-1}(\pi_{\ga}(Y))$.

If $Y' = s_{\ga} \cdot Y \neq Y$, one can consider the degree of the restricted morphism $\pi_{\ga}|_Y$ over its image.  This degree is always either $1$ or $2$.  We offer some brief explanation in the special case when $H$ is a \textit{symmetric subgroup}, i.e. the fixed points of an involution of $G$.  In this setting, we have the following terminology:  Either
\begin{enumerate}
	\item $\ga$ is ``complex" for $Q$; or
	\item $\ga$ is ``non-compact imaginary" for $Q$.
\end{enumerate}

The latter case breaks up into two subcases, known as ``type I" and ``type II".  These cases are differentiated by the $H$-orbit structure on the set $Z_{\ga}(Q)$ defined above.  In the ``type I" case, $Z_{\ga}(Q)$ is comprised of the dense orbit $Q'$, the orbit $Q$, and one other orbit $s_{\ga} \times Q$.  Here, $\times$ denotes the ``cross-action" of the Weyl group $W$ on $H \backslash G/B$, defined as
\[ w \times (H \cdot gB) = H \cdot gw^{-1}B. \]
In the ``type II" case, $Z_{\ga}(Q)$ is comprised simply of the dense orbit $Q'$ and the orbit $Q$, and in fact, $s_{\ga} \times Q = Q$.  In particular, if one knows that $\ga$ is non-compact imaginary for $Q$, then whether it is type I or type II depends only on whether $Q$ is fixed by the cross-action of $s_{\ga}$.

With all of this said, the result is as follows:
\begin{prop}
Suppose that $Y' = s_{\ga} \cdot Y \neq Y$.  If $\ga$ is complex or non-compact imaginary type I for $Q$, then $\pi_{\ga}|_Y$ is birational over its image.  If $\ga$ is non-compact imaginary type II for $Q$, then $\pi_{\ga}|_Y$ has degree 2 over its image.
\end{prop}
For a proof of this fact, and for more details on the definitions of complex, non-compact imaginary, etc., the reader may consult \cite{Richardson-Springer-90}.

The preceding terminology being particular to the symmetric case, we now return to the setting where $H$ is an arbitrary spherical subgroup.  In \cite{Brion-01}, the poset graph for the weak order on $H \backslash G/B$ is endowed with additional data as follows:  Whenever $Y' = s_{\ga} \cdot Y \neq Y$ and the degree of $\pi_{\ga}|_Y$ is $2$ (i.e. in the symmetric cases, whenever $\ga$ is non-compact imaginary type II for $Y$), $Y$ and $Y'$ are connected by a double edge.  Otherwise, the edge connecting $Y$ to $Y'$ is simple.  Each edge, whether simple or double, is also labeled with the appropriate simple root $\ga$, or perhaps an index $i$ if $s_{\ga} = s_{\ga_i}$ for a chosen ordering on the simple roots.

If $w \in W$, with $s_{i_1} \hdots s_{i_k}$ a reduced expression for $w$, set 
\[ w \cdot Y = s_{i_1} \cdot (s_{i_2} \cdot \hdots (s_{i_k} \cdot Y) \hdots ). \]
This is well-defined, independent of the choice of reduced expression for $w$, and defines an action of a certain monoid $M(W)$ on the set of $H$-orbit closures (\cite{Richardson-Springer-90}).  As a set, the monoid $M(W)$ is comprised of elements $m(w)$, one for each $w \in W$.  The multiplication on $M(W)$ is defined inductively by
\[ m(s)m(w) = 
\begin{cases}
	m(sw) & \text{ if $l(sw) > l(w)$}, \\
	m(w) & \text{ otherwise.}
\end{cases} \]

(We will use the notation $w \cdot Y$, as opposed to $m(w) \cdot Y$, to indicate this action, with the understanding that this defines an action of $M(W)$, and not of $W$.)

Suppose that $Y$ is an $H$-orbit closure on $G/B$ of codimension $d$.  Define the following subset of $W$:
\[ W(Y) := \{ w \in W \ \vert \ w \cdot Y = G/B \text{ and } l(w) = d\}. \]
(Note that in this definition, ``$G/B$" refers to the closure of the dense, open orbit.)  Elements of $W(Y)$ are precisely those $w$ such that there is a path connecting $Y$ to the top vertex of the weak order graph, the product of whose edge labels is $w$.  For any $w \in W(Y)$, denote by $D(w)$ the number of double edges in such a path.  (Although there may be more than one, any such path has the same number of double edges, so $D(w)$ \textit{is} well-defined.  See \cite[Lemma 5]{Brion-01}.)

We now recall a theorem of Brion, alluded to in the introduction, which we will ultimately use to obtain positive rules for Schubert structure constants.
\begin{theorem}[\cite{Brion-01}]\label{thm:brion}
With notation as above, in $H^*(G/B)$, the fundamental class of $Y$ is expressed in the Schubert basis as follows:
\[ [Y] = \displaystyle\sum_{w \in W(Y)} 2^{D(w)} S_w. \]
\end{theorem}

\section{$GL(p,\C) \times GL(q,\C)$-orbits on $GL(p+q,\C)/B$}\label{ssec:type-a}
We start with perhaps the most interesting example of a pair consisting of a classical group and a spherical Levi subgroup, the pair 
\[ (G,L)=(GL(n,\C),GL(p,\C) \times GL(q,\C)), \]
where $n=p+q$.  The Levi group $L$ in this case is symmetric.  We realize it as the subgroup of block diagonal matrices consisting of an upper-left $p \times p$ block and a lower-right $q \times q$ block.  This group is the Levi of the maximal parabolic $P$ which corresponds to the set of all simple roots less $\ga_p = x_p - x_{p+1}$.

\subsection{Richardson varieties stable under $L$}
The minimal (respectively, maximal) length coset representatives of $W/W_P$ are well-known to be the Grassmannian (respectively, anti-Grassmannian) permutations with unique descent (respectively, unique ascent) at $p$.  By Proposition \ref{prop:richardsons-stable-under-l}, Richardson varieties of the form $X_u^v$ with $u \geq v$ and $u$ (respectively, $v$) a maximal (respectively, minimal) length coset representative of $W_P \backslash W$ are stable under $L$.  Thus we shall be concerned with Richardsons of the form $X_u^v$ where $u$ is \textit{inverse} to an anti-Grassmannian permutation, and where $v$ is inverse to a Grassmannian permutation.  Such a $u$ has the property that its one-line notation is a ``shuffle" of the sets $\{p,\hdots,1\}$ and $\{n,\hdots,p+1\}$ --- that is, in the one-line notation, $p,\hdots,1$ appear in descending order, and $n,\hdots,p+1$ appear in descending order, but the two sets can be mixed together in any way.  Similarly, such a $v$ has the property that its one-line notation is a shuffle of the sets $\{1,\hdots,p\}$ and $\{p+1,\hdots,n\}$.

\subsection{Parametrizing $L \backslash G/B$}\label{sect:pq_example}
The finitely many $L$-orbits on $G/B$ are parametrized by \textit{$(p,q)$-clans}, as described in, e.g., \cite{Matsuki-Oshima-90,Yamamoto-97,McGovern-Trapa-09}.  We recall this parametrization in detail.  
\begin{definition}
A \textbf{$(p,q)$-clan} is a string of $n=p+q$ symbols, each of which is a $``+"$, a $``-"$, or a natural number.  The string must satisfy the following two properties:
\begin{enumerate}
	\item Every natural number which appears must appear exactly twice in the string.
	\item The difference in the number of plus signs and the number of minus signs in the string must be $p-q$.  (If $q > p$, then there should be $q-p$ more minus signs than plus signs.)
\end{enumerate}
\end{definition}

We only consider such strings up to an equivalence which says, essentially, that it is the positions of matching natural numbers, rather than the actual values of the numbers, which determine the clan.  So, for instance, the clans $(1,2,1,2)$, $(2,1,2,1)$, and $(5,7,5,7)$ are all the same, since they all have matching natural numbers in positions $1$ and $3$, and also in positions $2$ and $4$.  On the other hand, $(1,2,2,1)$ is a different clan, since it has matching natural numbers in positions $1$ and $4$, and in positions $2$ and $3$.

The set of $(p,q)$-clans is in bijection with the set of $L$-orbits on $G/B$.  Moreover, given a clan $\gamma$, the orbit $Q_{\gamma}$ admits an explicit linear algebraic description in terms of the combinatorics of $\gamma$.  Let $E_p = \C \left\langle e_1,\hdots,e_p \right\rangle$ be the span of the first $p$ standard basis vectors, and let $\widetilde{E_q} = \C \left\langle e_{p+1},\hdots,e_n \right\rangle$ be the span of the last $q$ standard basis vectors.  Let $\pi: \C^n \rightarrow E_p$ be the projection onto $E_p$.

For any clan $\gamma=(c_1,\hdots,c_n)$, and for any $i,j$ with $i<j$, define the following quantities:
\begin{enumerate}
	\item $\gamma(i; +) = $ the total number of plus signs and pairs of equal natural numbers occurring among $(c_1,\hdots,c_i)$;
	\item $\gamma(i; -) = $ the total number of minus signs and pairs of equal natural numbers occurring among $(c_1,\hdots,c_i)$; and
	\item $\gamma(i; j) = $ the number of pairs of equal natural numbers $c_s = c_t \in \N$ with $s \leq i < j < t$.
\end{enumerate}

For example, for the $(2,2)$-clan $\gamma=(1,+,1,-)$, we have that 
\begin{enumerate}
	\item $\gamma(i; +) = 0,1,2,2$ for $i=1,2,3,4$;
	\item $\gamma(i; -) = 0,0,1,2$ for $i=1,2,3,4$; and
	\item $\gamma(i;j) = 1,0,0,0,0,0$ for $(i,j) = (1,2), (1,3), (1,4), (2,3), (2,4), (3,4)$.
\end{enumerate}

With all of this notation defined, we have the following theorem on $L$-orbits on $G/B$:
\begin{theorem}[\cite{Matsuki-Oshima-90,Yamamoto-97}]\label{thm:orbit_description}
Suppose $p+q=n$.  For a $(p,q)$-clan $\gamma$, define $Q_{\gamma}$ to be the set of all flags $F_{\bullet}$ having the following three properties for all $i,j$ ($i<j$):
\begin{enumerate}
	\item $\dim(F_i \cap E_p) = \gamma(i; +)$
	\item $\dim(F_i \cap \widetilde{E_q}) = \gamma(i; -)$
	\item $\dim(\pi(F_i) + F_j) = j + \gamma(i; j)$
\end{enumerate}

For each $(p,q)$-clan $\gamma$, $Q_{\gamma}$ is nonempty and stable under $L$.  In fact, $Q_{\gamma}$ is a single $L$-orbit on $G/B$.

Conversely, every $L$-orbit on $G/B$ is of the form $Q_{\gamma}$ for some $(p,q)$-clan $\gamma$.  Hence the association $\gamma \mapsto Q_{\gamma}$ defines a bijection between the set of all $(p,q)$-clans and the set of $L$-orbits on $G/B$.
\end{theorem}

\begin{remark}
The parametrization of $L$-orbits on $G/B$ by $(p,q)$-clans was described first in \cite{Matsuki-Oshima-90}.  In that paper, no proof of the correctness of the parametrization is given, and the above linear algebraic description of $Q_{\gamma}$ does not appear.  Both the proof and the explicit description of $Q_{\gamma}$ appear in \cite{Yamamoto-97}.
\end{remark}

We also recall the following formula, given in \cite{Yamamoto-97}, for the dimension of the $L$-orbit $Q_{\gamma}$ in terms of the clan $\gamma$.  First define the ``length" of a clan $\gamma$ to be 
\begin{equation}
\label{e:length}
l(\gamma) = \sum_{c_i=c_j \in \bbN, i<j}\left( j-i-
\#\{k \in \bbN \; | \; c_s = c_t = k \text{ for some }
s < i<t<j \} \right ).
\end{equation}
Then
\begin{equation}
\label{e:orbit-dimension}
\dim(Q_{\gamma}) = 
d(L) +  l(\gamma),
\end{equation}
where $d(L)$ is the dimension of the flag variety for $L$, namely
$\frac12(p(p-1) + q(q-1))$.

Next, we describe the weak order on $L \backslash G/B$ in terms of this parametrization, following \cite{Yamamoto-97,McGovern-Trapa-09}.  Let $\frt$ be the Cartan subalgebra of $\text{Lie}(G) = \mathfrak{gl}(n,\C)$ consisting of diagonal matrices.  Let $x_i$ ($i=1,\hdots,n$) be coordinates on $\frt$, with
\[ x_i(\text{diag}(a_1,\hdots,a_n)) = a_i. \]
The simple roots are of the form $\ga_i = x_i - x_{i+1}$ ($i=1,\hdots,n-1$).  The root $\ga_i$ is complex for the orbit $Q_{\gamma}$ corresponding to $\gamma=(c_1,\hdots,c_n)$ if and only if $(c_i,c_{i+1})$ satisfy one of the following:
\begin{enumerate}
	\item $c_i$ is a sign, $c_{i+1}$ is a number, and the mate of $c_{i+1}$ occurs to the right of $c_{i+1}$;
	\item $c_i$ is a number, $c_{i+1}$ is a sign, and the mate of $c_i$ occurs to the left of $c_i$; or
	\item $c_i$ and $c_{i+1}$ are unequal natural numbers, and the mate of $c_i$ occurs to the left of the mate of $c_{i+1}$.
\end{enumerate}

In these cases, the orbit $s_{\ga_i} \cdot Q_{\gamma}$ is $Q_{\gamma'}$, where the clan $\gamma'$ is obtained from $\gamma$ by interchanging $c_i$ and $c_{i+1}$.

As examples of (1), (2), and (3) above, when $p=q=2$, we have
\begin{enumerate}
	\item $s_{\ga_1} \cdot (+,1,-,1) = (1,+,-,1)$; 
	\item $s_{\ga_2} \cdot (1,1,+,-) = (1,+,1,-)$; and
	\item $s_{\ga_2} \cdot (1,1,2,2) = (1,2,1,2)$.
\end{enumerate}

On the other hand, $\ga_i$ is non-compact imaginary for $Q_{\gamma}$ if and only if $(c_i,c_{i+1})$ are opposite signs.  In this case, $s_{\ga_i} \cdot Q_{\gamma} = Q_{\gamma''}$, where $\gamma''$ is obtained from $\gamma$ by replacing the signs in positions $(i,i+1)$ by matching natural numbers.  So, for instance, when $p=q=2$, $s_{\ga_2} \cdot (+,+,-,-) = (+,1,1,-)$.

The cross-action of $w \in S_n$ on any clan $\gamma$ (more precisely, on the orbit $Q_{\gamma}$) is the obvious one, by permutation of the characters of $\gamma$.  In particular, when $\ga_i$ is non-compact imaginary for $\gamma$, the cross-action of $s_{\ga_i}$ interchanges the opposite signs in positions $(i,i+1)$.  Thus for a non-compact imaginary root $\ga_i$, $s_{\ga_i} \times Q_{\gamma} \neq Q_{\gamma}$, and so we see that \textit{all non-compact imaginary roots are of type I}.  This establishes

\begin{prop}\label{prop:single_edges}
In the weak order graph for $L \backslash G/B$, all edges are single.
\end{prop}

\begin{remark}
The previous proposition follows from the discussion of the preceding paragraph, but can also be deduced using \cite[Corollary 2]{Brion-01}.  Indeed, this example is mentioned specifically in the discussion immediately following the statement of that corollary.
\end{remark}

Relative to the parametrization described here, the closed orbits (those minimal in the weak order) are those whose clans consist solely of $p$ plus signs and $q$ minus signs.  The dense open orbit is the one whose clan is $\gamma_0 := (1,2,\hdots,q-1,q,+,\hdots,+,q,q-1,\hdots,2,1)$ ($p-q$ plus signs appearing in the middle) if $p \geq q$, or $(1,2,\hdots,p-1,p,-,\hdots,-,p,p-1,\hdots,2,1)$ ($q-p$ minus signs appearing in the middle) if $q > p$.

With all of these combinatorics in hand, we recast the $M(W)$-action on $L$-orbits as a sequence of ``operations" on $(p,q)$-clans.  Let $\gamma=(c_1,\hdots,c_n)$ be a $(p,q)$-clan.  Given a simple root $s_i = s_{\ga_i}$, consider the following two possible operations on $\gamma$:
\begin{enumerate}[(a)]
	\item Interchange characters $c_i$ and $c_{i+1}$.
	\item Replace characters $c_i$ and $c_{i+1}$ by matching natural numbers.
\end{enumerate}

Then for $i=1,\hdots,n-1$,
\begin{enumerate}
	\item If $c_i$ is a sign, $c_{i+1}$ is a natural number, and the mate of $c_{i+1}$ occurs to the right of $c_{i+1}$, then $s_i \cdot \gamma$ is obtained from $\gamma$ by operation (a).
	\item If $c_i$ is a number, $c_{i+1}$ is a sign, and the mate of $c_i$ occurs to the left of $c_i$, then $s_i \cdot \gamma$ is obtained from $\gamma$ by operation (a).
	\item If $c_i$ and $c_{i+1}$ are unequal natural numbers, with the mate of $c_i$ occurring to the left of the mate for $c_{i+1}$, then $s_i \cdot \gamma$ is obtained from $\gamma$ by operation (a).
	\item If $c_i$ and $c_{i+1}$ are opposite signs, then $s_i \cdot \gamma$ is obtained from $\gamma$ by operation (b).
\end{enumerate}

If none of the above hold, then $s_i \cdot \gamma = \gamma$.

This extends in the natural way to an action of $M(W)$ on the set of all $(p,q)$-clans.  Note that if $Y_{\gamma} = \overline{Q_{\gamma}}$ is an $L$-orbit closure, the geometric condition that $w \cdot Y_{\gamma} = G/B$ is equivalent to the combinatorial condition that $w \cdot \gamma = \gamma_0$.

\subsection{The coincidence of $L$-orbit closures and Richardson varieties}
We know that every Richardson variety of the form $X_u^v$ with $u$ a shuffle of $p,\hdots,1$ and $n,\hdots,p+1$, and $v$ a shuffle of $1,\hdots,p$ and $p+1,\hdots,n$ is the closure of an $L$-orbit.  Now, we make this correspondence explicit by assigning the appropriate $(p,q)$-clan to such a pair $u,v$.  First, we note the following easy lemma on exactly when $u \geq v$ for such a pair $u,v$.

\begin{lemma}\label{lemma:bruhat-comparability}
Let $u,v$ be as described above.  For each $i$ between $1$ and $n$, define
\[ F(u,v,i) := \#\{j \mid 1 \leq j \leq i,\ u(j) > p, \text{ and } v(j) \leq p\}, \]
and
\[ S(u,v,i) := \#\{j \mid 1 \leq j \leq i,\ u(j) \leq p, \text{ and } v(j) > p\}. \]
Then $u \geq v$ if and only if $F(u,v,i) \geq S(u,v,i)$ for every $i=1,\hdots,n$.
\end{lemma}
\begin{proof}
We use the characterization of the Bruhat order given in Definition \ref{def:bruhat-2}.  Since $u$ is a shuffle of $p,\hdots,1$ and $n,\hdots,p+1$, for any $i \in [n]$, when $\{u(1),\hdots,u(i)\}$ is arranged in ascending order, the result is of the form
\[ \{s,s+1,\hdots,p,\vert,t,t+1,\hdots,n\}, \]
for some $s \leq p$ and for some $t > p$.  (The vertical bar marks the point at which the values change from being less than or equal to $p$ to being greater than $p$.)

Similarly, since $v$ is a shuffle of $1,\hdots,p$ and $p+1,\hdots,n$, when $\{v(1),\hdots,v(i)\}$ is arranged in ascending order, the result is
\[ \{1,2,\hdots,h,\vert,p+1,p+2,\hdots,k\}, \]
for some $h \leq p$ and $k > p$.

Comparing these sets, it is clear that the second set is element-wise less than or equal to the first if and only if the second set has at least as many elements which are less than or equal to $p$ as the first set does.  (Said another way, the vertical bar in the second set appears at least as far to the right as the vertical bar in the first set does.)  Defining
\[ P(u,v,i) := \{j \mid 1 \leq j \leq i,\ u(j) \leq p, \text{ and } v(j) \leq p\}, \]
we see that the number of elements less than or equal to $p$ in the first set is $P(u,v,i) + S(u,v,i)$, while the number of elements of the second set which are less than or equal to $p$ is $P(u,v,i) + F(u,v,i)$.  Thus we require that $P(u,v,i) + S(u,v,i) \leq P(u,v,i) + F(u,v,i)$, or that $S(u,v,i) \leq F(u,v,i)$.
\end{proof}

We now define a $(p,q)$-clan, denoted $\gamma(u,v)$, associated to such a pair $u,v$.

\begin{definition}
Given $u \geq v$ as above, define the $(p,q)$-clan $\gamma(u,v) = (e_1,\hdots,e_n)$ as follows, starting with $i=1$ and moving from left to right:
\begin{itemize}
	\item If $u(i),v(i) \leq p$, $e_i = +$;
	\item If $u(i),v(i) > p$, $e_i = -$;
	\item If $u(i) > p$, $v(i) \leq p$, $e_i$ is the \textit{first} occurrence of a natural number;
	\item If $u(i) \leq p$, $v(i) > p$, $e_i$ is the \textit{second} occurrence of the most recently occurring natural number which does not yet have a mate.
\end{itemize}
\end{definition}
For example, taking $p=q=3$ and the pair $u=365421$, $v=142356$, we see that $e_1=+$, $e_2=-$, $e_3=1$ (first occurrence), $e_4 = 2$ (first occurrence), $e_5 = 2$ (second occurrence of the $2$, the most recently appearing first occurrence without a mate), and $e_6=1$.  Thus $\gamma(u,v) = (+,-,1,2,2,1)$.

Note that Lemma \ref{lemma:bruhat-comparability} guarantees that this definition makes sense.  Indeed, when $u \geq v$, there are always at least as many occurrences of the third bullet above as of the fourth bullet as we move from left to right, so in an instance where $e_i$ is supposed to be a second occurrence of a natural number, we can always determine ``the most recently occurring natural number which does not yet have a mate."

\begin{theorem}\label{thm:k-orbits-richardson-varieties}
Suppose $u,v$ is a pair of permutations with $u \geq v$ and $X_u^v$ stable under $L$.  Let $\gamma = \gamma(u,v)$.  Then $X_u^v = \overline{Q_{\gamma}}$.
\end{theorem}
\begin{proof}
We know by the general observations of the introduction that $X_u^v$ has a dense $L$-orbit.  Since each $L$-orbit closure is irreducible ($L$ being connected), and since $X_u^v$ is irreducible, we need only observe the following:
\begin{enumerate}
	\item $Q_{\gamma} \subseteq X_u^v$, and
	\item $\dim(Q_{\gamma}) = \dim(X_u^v)$.
\end{enumerate}
For the first, suppose that $F_{\bullet}$ is a flag in $Q_{\gamma}$.  Then by Theorem \ref{thm:orbit_description}, we know that
\[ \dim(F_i \cap E_p) = \gamma(i; +) \]
for each $i=1,\hdots,n$.  Thus $F_{\bullet}$ belongs to some Schubert cell whose indexing permutation has the $p$th column of its rank matrix given by the numbers $\gamma(i; +)$ for $i=1,\hdots,n$.  We show that $u$ is the unique maximal element in the Bruhat order among all permutations of this type.  Note from the definition of $\gamma$ the way in which $u$ and $\gamma$ correspond:  In the one-line notation for $u$, the numbers $p,\hdots,1$ occur, in descending order, on the $+$'s and second occurrences of natural numbers of $\gamma$, and the numbers $n,\hdots,p+1$ occur, in descending order, on the $-$'s and first occurrences of natural numbers.  Define
\[ W^+ := \{ w \in W \ \vert \ r_w(i,p) = \gamma(i;+) \ \forall i \in [n]\}, \]
and consider permutations $w \in W^+$.  They are precisely those permutations whose rank matrices are of the form
\[
\begin{pmatrix}
	r_w(1,1) & \hdots & r_w(1,p) & \hdots & r_w(1,n) \\
	\vdots & \vdots & \vdots & \vdots & \vdots \\
	r_w(n,1) & \hdots & r_w(n,p) & \hdots & r_w(n,n)
\end{pmatrix}
=
\begin{pmatrix}
	* & \hdots & \gamma(1;+) & \hdots & * \\
	\vdots & \vdots & \vdots & \vdots & \vdots \\
	* & \hdots & \gamma(n;+) & \hdots & *
\end{pmatrix}
\]

By Definition \ref{def:bruhat-1} of the Bruhat order, to see that this set contains a unique maximal element, it suffices to show that the remaining entries of the rank matrix can be ``filled in" in a way which produces a rank matrix $R$ such that any other rank matrix having the prescribed $p$th column must be greater than or equal than $R$ in every single position.

The way to accomplish this is to place the jumps as far to the right as possible on every single row, starting with the first.  Set $\gamma(0;+) = 0$.  Then for any $i \in [n]$, either $\gamma(i;+) = \gamma(i-1;+)$, or $\gamma(i;+) = \gamma(i-1;+) + 1$.

If $\gamma(i;+) = \gamma(i-1;+)$, the jump in the $i$th row has not yet occurred by the point we reach the $p$th column.  We put it as far to the right as possible, meaning that the first time we encounter such a row, we place the jump in position $n$, the second time we put it in position $n-1$, etc.

If $\gamma(i;+) = \gamma(i-1;+) + 1$, then the jump in the $i$th row \textit{has} occurred by the $p$th column.  Again, we want to place the jump as far to the right as possible, so the first time we encounter such a row, we put the jump at $p$, the second time at $p-1$, etc.

This gives us the rank matrix of a permutation which assigns the numbers $n,n-1,\hdots,p+1$, in descending order, to those $i$ with $\gamma(i;+) = \gamma(i-1;+)$ (the coordinates of the $-$'s and first occurrences of natural numbers), and which assigns the numbers $p,p-1,\hdots,1$, in descending order, to those $i$ with $\gamma(i;+) = \gamma(i-1;+) + 1$ (the coordinates of the $+$'s and second occurrences of natural numbers).  As noted above, this is precisely the permutation $u$.  Thus $Q_{\gamma} \subseteq X_u$.

A completely analogous argument shows that $Q_{\gamma} \subset X^v$.  Indeed, given $F_{\bullet} \in Q_{\gamma}$, we have by Theorem \ref{thm:orbit_description} that
\[ \dim(F_i \cap \widetilde{E_q}) = \gamma(i; -) \]
for $i=1,\hdots,n$.

Define
\[ W^- := \{w \in W \ \vert \ r_{w_0w}(i,q) = \gamma(i;-) \ \forall i \in [n]\}. \]
We know that $F_{\bullet}$ is in an opposite Schubert cell $X^w$ for $w \in W^-$.  We want to show that $v$ is the unique minimal element of $W^-$, or, equivalently, that $w_0v$ is the unique maximal element of $w_0W^-$.  Note how $v$ and $\gamma$ correspond:  the one-line notation for $v$ has the numbers $1,\hdots,p$ occurring in order on the $+$'s and first occurrences of natural numbers of $\gamma$, and the numbers $p+1,\hdots,n$ occurring in order on the $-$'s and second occurrences of natural numbers of $\gamma$.  Thus $w_0v$ has the numbers $n,\hdots,q+1$ occurring in order on the $+$'s and first occurrences of natural numbers, and $q,\hdots,1$ on the $-$'s and second occurrences of natural numbers.  Now, arguing just as we did above for $u$, one sees that $w_0v$ is the unique maximal element of the set of all permutations whose rank matrices have $q$th column given by the numbers $\gamma(i; -)$ for $i=1,\hdots,n$.  This is precisely the set $w_0W^-$.  Thus $w_0v$ is the unique maximal element of $w_0W^-$ or, equivalently, $v$ is the unique minimal element of $W^-$.  This establishes that $Q_{\gamma} \subset X^v$ and, combining with the above, that $Q_{\gamma} \subset X_u^v$.

Now, we have only to argue that the dimensions of $X_u^v$ and $Q_{\gamma}$ are the same.  The dimension of $X_u^v$ is $l(u) - l(v)$, as noted in Proposition \ref{prop:basic_richardson_varieties}.  On the other hand, if $\gamma=(c_1,\hdots,c_n)$, then the dimension of $Q_{\gamma}$ is given by Equation (\ref{e:orbit-dimension}) to be
\[ \frac{1}{2}(p(p-1) + q(q-1)) + \sum_{c_i=c_j \in \bbN, i<j}\left( j-i-
\#\{k \in \bbN \; | \; c_s = c_t = k \text{ for some }
s < i<t<j \} \right ). \]

Note, however, that the clan $\gamma$ ``avoids the pattern $(1,2,1,2)$", by which we mean that any two pairs of matching natural numbers are either nested, or disjoint.  This is a consequence of the way in which we defined $\gamma$ --- every time we see the second occurrence of a natural number, it is always the mate for the most recent unmated first occurrence to appear.  This means that for this particular clan $\gamma$, the term
\[ \#\{k \in \bbN \; | \; c_s = c_t = k \text{ for some }
s < i<t<j \} \]
is zero for every pair $c_i = c_j \in \bbN$, so our dimension formula simplifies to
\[  \frac{1}{2}(p(p-1) + q(q-1)) + \sum_{c_i=c_j \in \bbN, i<j}\left( j-i \right ). \]

Our task is to see that this number is equal to $l(u) - l(v)$.  We think of the length of a permutation $w$ as its number of inversions, i.e. the number of pairs $i,j$ with $i<j$ and $w(i) > w(j)$.  Since $u$ is a shuffle of $p,\hdots,1$ and $n,\hdots,p+1$, it has $\frac{1}{2}p(p-1)$ inversions coming from the numbers $p,\hdots,1$ being in reverse order, $\frac{1}{2}q(q-1)$ inversions coming from the numbers $n,\hdots,p+1$ being in reverse order, and some number of other inversions coming from the shuffling of the two sets together.  The permutation $v$, being a shuffle of $1,\hdots,p$ and $p+1,\hdots,n$, has \textit{all} of its inversions occurring as a result of the two sets being shuffled together.  Thus
\[ l(u) - l(v) = \]
\[ \frac{1}{2}(p(p-1) + q(q-1)) + \#\{i<j \mid u(i) \geq p+1, u(j) \leq p\} - \#\{i<j \mid v(i) \geq p+1, v(j) \leq p\}. \]
So we have reduced to showing that 
\[ \#\{i<j \mid u(i) \geq p+1, u(j) \leq p\} - \#\{i<j \mid v(i) \geq p+1, v(j) \leq p\} = \sum_{c_i=c_j \in \bbN, i<j}\left( j-i \right ). \]

If $\gamma = (c_1,\hdots,c_n)$, consider the possible values of $(c_i,c_j)$ ($i<j$) which imply that $u(i) \geq p+1$ and $u(j) \leq p$, or that $v(i) \geq p+1$ and $v(j) \leq p$.  For $u$, the possibilities are
\begin{enumerate}
	\item $(-,+)$
	\item $(-,S)$ (where $S$ indicates the second occurrence of some natural number)
	\item $(F,+)$ (where $F$ indicates the first occurrence of some natural number)
	\item $(F,S)$
\end{enumerate}

For $v$, they are
\begin{enumerate}
	\item $(-,+)$
	\item $(-,F)$
	\item $(S,+)$
	\item $(S,F)$
\end{enumerate}

For short, we denote $\#\{i<j \mid c_i = \text{`$-$' and } c_j = \text{`$+$'}\}$ by $\gamma_{-,+}$, and similarly for the other possible pairs listed above.  Thus 
\[ \#\{i<j \mid u(i) \geq p+1, u(j) \leq p\} - \#\{i<j \mid v(i) \geq p+1, v(j) \leq p\} = \]
\[ (\gamma_{-,+} + \gamma_{-,S} + \gamma_{F,+} + \gamma_{F,S}) - (\gamma_{-,+} + \gamma_{-,F} + \gamma_{S,+} + \gamma_{S,F}) = \]
\[ (\gamma_{-,S} - \gamma_{-,F}) + (\gamma_{F,+} - \gamma_{S,+}) + (\gamma_{F,S} - \gamma_{S,F}). \]

Consider first the quantity $\gamma_{-,S} - \gamma_{-,F}$.  We claim that
\[ \gamma_{-,S} - \gamma_{-,F} = \sum_{c_i=c_j \in \bbN, i<j} \#\{c_k = - \mid i < k < j\}. \]
Indeed, $(-,S)$ pairs can appear in two different ways:
\begin{enumerate}
	\item $(\hdots,1,\hdots,-^*,\hdots,1^*,\hdots)$, or
	\item $(\hdots,-^*,\hdots,1,\hdots,1^*,\hdots)$.
\end{enumerate}
The asterisks above mark the positions of the $-$ and $S$ characters under consideration.  Note that the $(-,S)$ pairs of type (2) are in $1$-to-$1$ correspondence with the $(-,F)$ pairs.  Indeed, the $(-,S)$ pair depicted in (2) above corresponds to the $(-,F)$ pair consisting of the same $-$ sign and the first occurrence of the $1$.  Conversely, every $(-,F)$ pair corresponds to the $(-,S)$ pair consisting of the same $-$ sign and the second occurrence of the number.  Thus to compute $\gamma_{-,S} - \gamma_{-,F}$, we can simply disregard $(-,S)$ pairs of type (2) above and $(-,F)$ pairs, and count the number of $(-,S)$ pairs of type (1) above.  To count the $(-,S)$ pairs of type (1) we count, for each matching pair of numbers, the number of $-$ signs enclosed by that pair.  This is precisely the sum given above.

An identical argument shows that 
\[ \gamma_{F,+} - \gamma_{S,+} = \sum_{c_i=c_j \in \bbN, i<j} \#\{c_k = + \mid i < k < j\}. \]

Finally, consider $\gamma_{F,S} - \gamma_{S,F}$.  We claim that 
\[ \gamma_{F,S} - \gamma_{S,F} = \sum_{c_i=c_j \in \bbN, i<j} \left( \#\{c_k \in \bbN \mid i < k < j\} + 1 \right). \]
Indeed, consider the ways in which $(F,S)$ pairs can appear.  First, they could be a pair of matching natural numbers, as in 
\begin{enumerate}
	\item $(\hdots,1^*,\hdots,1^*,\hdots)$.
\suspend{enumerate}
As before, we use asterisks to mark the positions of the characters under consideration.  If the $F$ and $S$ are not a matching pair, then in light of the aforementioned $(1,2,1,2)$-avoidance, there are three more possibilities:
\resume{enumerate}
	\item $(\hdots,1^*,\hdots,2,\hdots,2^*,\hdots,1,\hdots)$;
	\item $(\hdots,1,\hdots,2^*,\hdots,2,\hdots,1^*,\hdots)$;
	\item $(\hdots,1^*,\hdots,1,\hdots,2,\hdots,2^*,\hdots)$.
\end{enumerate}

Now, note that the $(F,S)$ pairs of type (4) are in $1$-to-$1$ correspondence with $(S,F)$-pairs.  Indeed, the $(F,S)$ pair of type (4) depicted above corresponds to the $(S,F)$ pair $(\hdots,1,\hdots,1^*,\hdots,2^*,\hdots,2,\hdots)$, formed by the second occurrence of the $1$ followed by the first occurrence of the $2$.  Conversely, each $(S,F)$ pair arises due to two pairs of numbers forming the pattern $(1,1,2,2)$, and thus corresponds to the $(F,S)$ pair formed by the first occurrence of the $1$ followed by the second occurrence of the $2$.  Thus to compute $\gamma_{F,S} - \gamma_{S,F}$, we can simply disregard $(F,S)$ pairs of type (4) and $(S,F)$ pairs, and count $(F,S)$ pairs of types (1), (2), and (3).  To count pairs of types (2) and (3) combined, we can simply count twice the number of occurrences of the pattern $(1,2,2,1)$ within the clan.  Equivalently, for each pair of matching natural numbers, we can count the natural numbers flanked by that pair, giving us the sum
\[ \sum_{c_i=c_j \in \bbN, i<j} \#\{c_k \in \bbN \mid i < k < j\}. \]
To add in $(F,S)$ pairs of type (1), we simply count the number of pairs of natural numbers.  This then gives
\[ \sum_{c_i=c_j \in \bbN, i<j} \left(\#\{c_k \in \bbN \mid i < k < j\} + 1\right), \]
as claimed.

Adding up $\gamma_{-,S} - \gamma_{-,F}$, $\gamma_{F,+} - \gamma_{S,+}$, and $\gamma_{F,S} - \gamma_{S,F}$, we get
\[ \sum_{c_i=c_j \in \bbN, i<j} \left( \#\{c_k = - \mid i < k < j\} + \#\{c_k = + \mid i < k < j\} + \#\{c_k \in \bbN \mid i < k < j\} + 1 \right) = \]
\[ \sum_{c_i=c_j \in \bbN, i<j} \left( j - i \right), \]
as desired.  This completes the proof.
\end{proof}

\begin{remark}
As noted in the proof of Theorem \ref{thm:k-orbits-richardson-varieties}, any Richardson variety $X_u^v$ stable under $L$ is the closure of an $L$-orbit corresponding to a clan $\gamma$ which avoids the pattern $(1,2,1,2)$.  In fact, it is clear that the converse is true.  Given such a clan $\gamma$, the orbit closure $\overline{Q_{\gamma}}$ is equal to $X_{u(\gamma)}^{v(\gamma)}$ for the permutations $u(\gamma)$, $v(\gamma)$ whose one-line notations are obtained as follows:
\begin{itemize}
	\item For $u(\gamma)$, place $p,\hdots,1$ (in descending order) on the $+$'s and second occurrences of natural numbers of $\gamma$, and $n,\hdots,p+1$ (in descending order) on the $-$'s and first occurrences of natural numbers.
	\item For $v(\gamma)$, place $1,\hdots,p$ (in ascending order) on the $+$'s and first occurrences of natural numbers of $\gamma$, and $p+1,\hdots,n$ (in ascending order) on the $-$'s and second occurrences of natural numbers. 
\end{itemize}
Using this observation, one can draw some interesting conclusions on properties of $L$-orbit closures in terms of the combinatorial properties of their clans.  For instance, most of the pattern-avoidance criteria for (rational) smoothness of an $L$-orbit closure given in \cite{McGovern-09a} can be recovered by restricting attention to the $(1,2,1,2)$-avoiding case and using known combinatorial characterizations of the singular loci of Schubert varieties.  This will be discussed in future work, joint with Alexander Woo.
\end{remark}

Taken together, Theorem \ref{thm:k-orbits-richardson-varieties}, Theorem \ref{thm:brion}, and the combinatorics laid out in Subsection \ref{sect:pq_example} give a positive (indeed, multiplicity-free) rule for structure constants $c_{w_0u,v}^w$ when $u$ is inverse to an anti-Grassmannian permutation with unique ascent at $p$, $v$ is inverse to a Grassmannian permutation with unique descent at $p$, and $u \geq v$.

\begin{theorem}\label{thm:structure_constants}
Let $\gamma_0$ be the clan parametrizing the open dense $L$-orbit on $G/B$, as in Subsection \ref{sect:pq_example}.  Suppose that $X_u^v$ is a Richardson variety stable under the action of $L$ (so that $u,v$ are as in the statement of Lemma \ref{lemma:bruhat-comparability}).  Then for $w$ of the appropriate length (namely $l(w) = l(w_0u) + l(v)$),
\[ 
c_{w_0u,v}^w = 
\begin{cases}
	1 & \text{ if $w \cdot \gamma(u,v) = \gamma_0$}, \\
	0 & \text{ otherwise.}
\end{cases}
\]
\end{theorem}
\begin{proof}
Let $Y=\overline{Q_{\gamma(u,v)}}$.  By Theorem \ref{thm:k-orbits-richardson-varieties}, along with equation (\ref{eqn:classes_of_richardson_varieties}), we have 
\[ [Y] = [X_u^v] = S_{w_0u} \cdot S_v, \]
so the structure constants $c_{w_0u,v}^w$ are identically the coefficients of the various $S_w$ in the Schubert basis expansion of $[Y]$.

The fact that all such coefficients are $0$ or $1$ follows from Proposition \ref{prop:single_edges} and Theorem \ref{thm:brion}.  Note that requiring that $w$ be ``of the appropriate length" (i.e. requiring that $S_w$ live in the only degree it could in order for $c_{w_0u,v}^w$ to be non-zero) is equivalent to requiring that $l(w) = \text{codim}(Y)$, as we do in the definition of $W(Y)$ prior to the statement of Theorem \ref{thm:brion}.  Indeed, the codimension of $Y$ is precisely
\[ \dim(G/B) - \dim(X_{u}^v) = l(w_0u) + l(v). \]
Thus by Theorem \ref{thm:brion}, the only other requirement we must impose on $w$ for $c_{w_0u,v}^w = 1$ is that $w \cdot Y = G/B$.  As was noted at the end of Subsection \ref{sect:pq_example}, this is equivalent to the combinatorial condition that $w \cdot \gamma(u,v)=\gamma_0$.
\end{proof}

\begin{example}\label{ex:example-1}
Consider the Schubert product $S_u \cdot S_v$ with $u = 31425$, $v=14253$.  This is the class of the Richardson variety $X_{35241}^{14253}$, which corresponds to the $(3,2)$-clan $\gamma(35241,14253) = (+,-,+,-,+)$.  We have $l(31425) = l(14253) = 3$, and there are $20$ elements of $S_5$ of length $6$.  Table 1 of the Appendix shows each of these $20$ elements as words in the simple reflections, the clan obtained from computing the action of each on the clan $\gamma = (+,-,+,-,+)$, and the corresponding structure constant $c_{u,v}^w$ specified by Theorem \ref{thm:structure_constants}.
\end{example}

\section{Remarks on Other Cases}
\subsection{Symmetric Cases}\label{sec:symmetric}
As noted in the introduction, there are other pairs $(G,L)$ consisting of a simple classical group and a spherical Levi subgroup.  In this section, we make some brief remarks and offer roughly stated conjectures applying to the pairs $(SO(2n+1,\C),\C^* \times SO(2n-1,\C))$ (type $B$) and $(SO(2n,\C),\C^* \times SO(2n-2,\C))$ (type $D$).  The few specifics that we give here apply to the type $B$ case, for simplicity's sake.  They extend in an obvious way to the type $D$ case, with some extra notation.

Here, $L$ is the Levi factor of the maximal parabolic subgroup $P$ of $G$ corresponding to omission of the simple root $\ga_1 = x_1 - x_2$.  Viewing elements of the Weyl group $W$ as signed permutations of $\{1,\hdots,n\}$, the $L$-stable Richardsons are of the form $X_u^v$ where $u$ has $\overline{2},\hdots,\overline{n}$ appearing in order in the one-line notation, $v$ has $2,\hdots,n$ appearing in order in the one-line notation, and where $u \geq v$.  (The requirement that $u \geq v$ turns out to correspond to a requirement on the relative positions of the $1$ or $\overline{1}$ in the one-line notations for $u$ and $v$.)

By the general remarks of the introduction, for such a pair $u,v$, the Richardson variety $X_u^v$ coincides with an $L$-orbit closure.  One could turn this into a Schubert calculus rule given an explicit combinatorial model for the $L$-orbits, along with an understanding of its weak order.  Alas, unlike in the case we considered in type $A$, no combinatorial model for these orbits has been described in the literature.  What \textit{has} been described (\cite{Matsuki-Oshima-90}) is a combinatorial model for the orbits of the symmetric subgroup $K=S(O(2,\C) \times O(2n-1,\C))$.  As explained in \cite{Matsuki-Oshima-90}, the $K$-orbits are parametrized by ``symmetric" $(2,2n-1)$-clans, and the weak order is completely understood in terms of this parametrization.  (We remark that the notation of \cite{Matsuki-Oshima-90} is a bit different, in that the symbols used in that paper more efficiently give only the first half of the $(2,2n-1)$-clan, with the second half being determined from the first by symmetry.  The translation between the two notations, and the corresponding reinterpretation of the weak order, is described in detail in \cite{Wyser-Thesis}.)

Note that $K$ is disconnected, having two components, and that $L=K^0$.  This means that some $K$-orbits will be disconnected, and will split as a union of two $L$-orbits, while others will be connected, and will thus coincide with a single $L$-orbit.  Thus a parametrization of the $L$-orbits can be obtained from the known parametrization of the $K$-orbits so long as we understand precisely which $K$-orbits are disconnected.  Preferably, this could be stated as a combinatorial condition on the clan parametrizing the $K$-orbit.  Alas, the author has only a conjecture on this matter (attributed to Peter Trapa).  That conjecture has been verified using ATLAS (available at \url{http://www.liegroups.org/}) through moderately high rank.

Given a proof of this conjecture, one can easily give a description of the weak order on $L \backslash G/B$ using the (known) weak order on $K \backslash G/B$.  Assuming Trapa's conjecture on $L \backslash G/B$ to be valid and taking this conjectural weak order as a starting point, one can then identify the $L$-stable Richardson varieties described above with the appropriate symmetric $(2,2n-1)$-clans and deduce Schubert calculus rules from this association, just as we did in the type $A$ case in Section \ref{ssec:type-a}.  In type $B$, this leads to the conjecture that for suitable $u,v$, and for $w$ of the appropriate length, all Schubert constants $c_{u,v}^w$ are either $0$, $1$, or $2$, with $c_{u,v}^w = 0$ only if $w \cdot \gamma(u,v) \neq \gamma_0$; $c_{u,v}^w = 2$  only if $w \cdot \gamma(u,v) = \gamma_0$ and the computation of the action of $w$ on $\gamma(u,v)$ involves a certain specific operation; and $c_{u,v}^w = 1$ otherwise.  (As in Section \ref{ssec:type-a}, here $\gamma(u,v)$ denotes the clan identified with the Richardson variety $X_u^v$, and $\gamma_0$ denotes the clan parametrizing the open, dense $L$-orbit on $G/B$.)  In type $D$, the situation is very similar, but the $M(W)$-action is defined a bit differently, and all relevant Schubert constants turn out to be either $0$ or $1$.

The interested reader can find the specific details of these conjectures in \cite{Wyser-13}.
	
\subsection{Non-symmetric Cases}\label{sec:non-symmetric}
	The remaining two pairs $(G,L)$ consisting of a classical group and spherical Levi subgroup not yet considered in this paper or in \cite{Wyser-11b} are $(Sp(2n,\C),\C^* \times Sp(2n-2,\C))$ and $(SO(2n+1,\C),GL(n,\C))$.  The former pair is quite similar to the pairs mentioned in the previous subsection, and indeed the Richardson varieties stable under that Levi are precisely the same as those described in the type $B$ case of the previous subsection.  Likewise, the latter pair is very similar to the pairs considered in \cite{Wyser-11b}, and the Richardson varieties stable under \textit{that} Levi bear precisely the same descriptions as those considered in the type $C$ case of \cite{Wyser-11b}.  However, combinatorial descriptions of $L \backslash G/B$ for these cases have not appeared in the literature, and are not (as far as the author can tell) easily deducible from, say, the results of \cite{Matsuki-Oshima-90}.  This is because, unlike the cases considered to this point, these Levis are \textit{not} (identity components of) symmetric subgroups.
	
	This complicates matters a bit if one wishes to prove the correctness of a combinatorial model for the $L$-orbits.  In the type $B$ case above, for example, a natural guess might be that the $L$-orbits are parametrized by some family of ``skew-symmetric" $(n,n)$-clans, as in the corresponding cases in types $C$ and $D$.  If one could translate such a clan into a linear algebraic description of the corresponding orbit as a set of flags, then it should be very easy to check that that set is stable under $L$.  However, it may be more difficult to prove that $L$ is actually transitive on the set.  In the cases of symmetric subgroups, one can get around this difficulty by applying some known combinatorial machinery underlying the ATLAS software, explained in \cite{Adams-DuCloux-09}.  However, this machinery is particular to the symmetric case.  And we have said nothing of the weak order or the appropriate placement of double edges in the weak order graph.
	
	Despite these difficulties, a natural guess might be that the Schubert calculus rules associated to $L$-stable Richardsons are the same as those alluded to in Subsection \ref{sec:symmetric} and those described in detail in \cite{Wyser-11b}.  Experimentation seems to suggest that for the pair $(Sp(2n,\C),\C^* \times Sp(2n-2,\C))$, and for relevant choices of $u,v$, the set 
	\[ \{w \in W \mid c_{u,v}^w \neq 0\} \]
	is identical to the analogous set for the type $B$ pair $(SO(2n+1,\C),\C^* \times SO(2n-1,\C))$ discussed in the previous subsection.  Likewise, for the pair $(SO(2n+1,\C),GL(n,\C))$, the relevant Schubert constants $c_{u,v}^w$ are non-zero for precisely the same $w$ as in the type $C$ case $(Sp(2n,\C),GL(n,\C))$ discussed in \cite{Wyser-11b}.  However, even in small rank, one finds examples where these Schubert constants are not the same.  This indicates that although the weak Bruhat interval $[Y,G/B]$ associated to an $L$-orbit closure $Y$ in these other cases may be isomorphic (as a graph with labeled edges) to the corresponding weak Bruhat interval in the known cases, there are some subtle differences in where double edges are placed.

\newpage

\begin{table}[h]
	\caption{Example \ref{ex:example-1}:  Computing the $(3,2)$ Schubert product $S_{31425} \cdot S_{14253}$}
	\begin{tabular}{|c|c|c|}
		\hline
		Length $6$ Element $w$ & $w \cdot (+,-,+,-,+)$ & $c_{u,v}^w$ \\ \hline
		$[4, 3, 2, 4, 3, 4]$ & $(+,1,2,2,1)$  & $0$ \\ \hline
		$[1, 3, 2, 4, 3, 4]$ & $(1,+,2,2,1)$ & $0$ \\ \hline
		$[1, 4, 3, 2, 3, 4]$ & $(1,+,2,2,1)$ & $0$ \\ \hline
		$[1, 4, 3, 2, 4, 3]$ & $(1,+,2,2,1)$ & $0$ \\ \hline
		$[2, 1, 2, 4, 3, 4]$ & $(1,2,2,+,1)$ & $0$ \\ \hline
		$[2, 1, 3, 2, 3, 4]$ & $(1,2,+,2,1)$ & $1$ \\ \hline
		$[2, 1, 3, 2, 4, 3]$ & $(1,2,+,2,1)$ & $1$ \\ \hline
		$[2, 1, 4, 3, 2, 4]$ & $(1,2,+,2,1)$ & $1$ \\ \hline
		$[2, 1, 4, 3, 2, 3]$ & $(1,2,2,+,1)$ & $0$ \\ \hline
		$[3, 2, 1, 4, 3, 4]$ & $(1,2,+,1,2)$ & $0$ \\ \hline
		$[3, 2, 1, 2, 3, 4]$ & $(1,2,+,2,1)$ & $1$\\ \hline
		$[3, 2, 1, 2, 4, 3]$ & $(1,2,+,2,1)$ & $1$ \\ \hline
		$[3, 2, 1, 3, 2, 4]$ & $(1,2,+,1,2)$ & $0$ \\ \hline
		$[3, 2, 1, 3, 2, 3]$ & $(1,2,2,1,+)$ & $0$ \\ \hline
		$[3, 2, 1, 4, 3, 2]$ & $(1,2,+,2,1)$ & $1$ \\ \hline
		$[4, 3, 2, 1, 3, 4]$ & $(1,2,+,2,1)$ & $1$  \\ \hline
		$[4, 3, 2, 1, 4, 3]$ & $(1,2,+,2,1)$ & $1$ \\ \hline
		$[4, 3, 2, 1, 2, 4]$ & $(1,+,2,2,1)$ & $0$ \\ \hline
		$[4, 3, 2, 1, 2, 3]$ & $(1,2,2,+,1)$ & $0$\\ \hline
		$[4, 3, 2, 1, 3, 2]$ & $(1,2,2,+,1)$ & $0$ \\ \hline
	\end{tabular}
\end{table}

\bibliographystyle{alpha}
\bibliography{../sourceDatabase}

\end{document}